\documentclass[english,letterpaper]{article}
\usepackage[english]{babel}
\usepackage{amsmath,amsthm,amssymb}

\usepackage{latexsym}
\usepackage{enumerate}

\newtheorem{theorem}{Theorem}
\newtheorem{lemma}{Lemma}
\newtheorem{proposition}{Proposition}

\newtheorem{definition}{Definition}

\newtheorem{remark}{Remark}
\newtheorem{example}{Example}

\title{Entropy, pressure, ground states and calibrated sub-actions for linear dynamics}

\author{Artur O. Lopes\thanks{Institute of Mathematics and Statistics - Federal University of Rio Grande do Sul. Partially supported by CNPq.}\; and Victor Vargas\thanks{School of Mathematics - National University of Colombia. Supported by FFJC-MINCIENCIAS Process 80740-628-2020.}
}

\begin{document}

\maketitle

{\bf Abstract:}
Denote by $X$ a  Banach space and by $T : X \to X$ a bounded linear operator  with non-trivial kernel satisfying suitable conditions.  We consider the concepts of entropy - for $T$-invariant probability measures - and pressure for H\"older continuous potentials. We also prove the existence  of ground states (the limit when temperature goes to zero) associated with such class of potentials when the Banach space $X$ is equipped with a Schauder basis. We produce an example concerning  weighted shift operators defined on the Banach spaces $c_0(\mathbb{R})$ and $l^p(\mathbb{R})$, $1 \leq p < +\infty$,  where our results do apply. In addition, we prove the existence  of calibrated sub-actions when the potential satisfies certain regularity conditions using properties of  the so-called  Ma\~n\'e potential. We also exhibit examples of selection at zero temperature and explicit sub-actions in the class of H\"older continuous potentials.

\vspace{2mm}


\vspace*{2mm}


\section{Introduction}

Ergodic Optimization is the branch of mathematics dedicated to the study of the properties of the set of invariant probability measures that maximize the value of the integral with respect to a fixed (at least continuous) potential $A$  (see for instance \cite{MR3114331}, \cite{MR3701349}, \cite{MR2191393} and \cite{J2}).  The underlying dynamics of the above-mentioned papers is given by the shift map,  which is the classical Ergodic Optimization framework.


Here we are interested in discrete-time Linear Dynamical Systems acting on a separable Banach space $X$. We point out that the dynamics of a bounded linear operator has some features which  are quite different from the dynamics of the  shift. For instance, if $v\in X$ is a  periodic point for $T:X \to X$, then, the collection of all vectors in the one-dimensional subspace generated by $v$ are also  periodic points. Nowadays  there exists  substantial work on the dynamics of bounded linear operators acting on Banach spaces (see for instance \cite{MR2533318, MR3759568, BeMe20, God, MR3255465, MR2919812}).

In  Statistical Mechanics the influence of the temperature $\text{Temp}>0$ is described by considering the potential $\frac{1}{\text{Temp}} \,A$. It is common to introduce the parameter $t= \frac{1}{\text{Temp}}$. The Gibbs state associated with the potential $t\,A$ will be denoted by $\mu_{t\,A}$. When the potential is of H\"older class several nice properties can be derived for  its corresponding Gibbs states (see for instance \cite{MR1085356}). In \cite{LMSV19} the authors show the existence of $T$-invariant probabilities with full support using the Gibbs state point of view.

Given a H\"older potential $A$ we call equilibrium state for $A$ an invariant one  that maximizes a variational principle of pressure (to be defined later). For the case when the dynamics is given by the shift several papers addressed the question of showing that  a Gibbs state for $A$ is an equilibrium state for $A$  as well (see for instance \cite{MR1085356}). In the first part  of our paper we will be interested in equilibrium states (they are invariant by the dynamics governed by a bounded surjective linear operator $T$ with finite dimensional non-trivial kernel). The second part of the paper  is dedicated to the topic of Ergodic Optimization.

Several results in Ergodic Optimization were developed  using  tools of  Thermodynamical Formalism for compact and Polish metric spaces (see for instance \cite{MR3114331, BGT, MR3701349, GT, MR2191393, MR2279266, MR2354972}).
This is so because maximizing probabilities for a potential $A$ appear in  a natural way as the limit of equilibrium states when the temperature $\text{Temp} = \frac{1}{t}$ goes to zero (which is the same to say  that $t \to +\infty$). An invariant probability obtained as an accumulation point  (as the limit of a subsequence $t_n \to +\infty$) is called a {\it ground state}. They are special in some sense because in addition, they maximize the entropy among all the maximizing probability measures of the potential $A$ (see \cite{MR3114331, MR1855838, MR3864383, MR2151222, MR2800665}).

  In the case of the uniqueness of such accumulation point, we say that there exists {\it selection of probability at zero temperature}. This is the case for example  when the maximizing probability is unique. One can show that for a generic H\"older potential, the maximizing probability is unique (see for instance \cite {MR1855838}).  A very important result in the area shows that generically in  the class of H\"older continuous potentials, the maximizing probability has support in a unique periodic orbit (see \cite{CoGro}).

In \cite{MR1958608}, \cite{MR2818689} and \cite{MR2176962}, it  was proved (without assuming uniqueness of the maximizing probability) the  existence of a unique accumulation point at zero temperature in the context of subshifts of finite type under the assumption that the potential depends only on finite coordinates. Se also  \cite{BGT} and \cite{GT} for another kind of examples where there exists selection at zero temperature.

In \cite{MR2864625} and \cite{van}, the authors present examples of a H\"older continuous potential where selection  at zero temperature does not occur.

For the  shift map in the non-compact setting, similar results were proved  showing the existence of ground states (the existence of  accumulation points at zero temperature).  Results in this direction were obtained for countable Markov shifts satisfying the so-called   BIP property, topologically transitive countable Markov shifts, and full shifts defined on the lattice $\mathbb{R}^{\mathbb{N}}$ (see for instance \cite{MR3864383, MR2151222, MR2800665, MR3377291, 1LoVa19, SV20}).

The study of sub-actions and calibrated sub-actions provides important tools in the study of maximizing measures associated with  H\"older continuous (or even continuous) potentials and its corresponding supports, This is so because a sub-action provides tools to identify the support of  maximizing probabilities (see \cite{MR3114331,MR1855838}).  Several works addressed this issue in the context of $X Y$ models, expanding maps of the circle, subshifts of finite type, and even in non-compact settings such as shifts defined on Polish spaces. A helpful  tool used to find sub-actions associated with a potential satisfying certain regularity is the so-called  Ma\~n\'e potential  (see for details \cite{MR2864625, ChFr19, MR1855838, MR3701349, GL1, MR2354972, MR3377291, LMST})

The dynamics of bounded linear operators on Banach spaces (see \cite{MR2533318, MR3759568, Gil, MR2919812}) present some special properties which are significantly different from the ones for the shift and continuous maps acting on Polish spaces. Questions  of topological nature as expansivity, shadowing, transitivity, and structural stability in this linear setting were addressed in \cite{MR3759568, BeMe20, Gil}. We refer the reader to the first part of \cite{LMSV19} for a short account of some basic definitions and results concerning the  theory of Linear Dynamics on  Banach spaces.

For the discrete-time dynamical action of a bounded linear operator $T:X \to X$ on a Banach space $X$, the paper \cite{MR3255465} present results about the existence of $T$-invariant probability measures with full support in the case that $X$ is reflexive and separable.  An extension of this result for a more general setting, including non-reflexive separable Banach spaces, was presented in \cite{LMSV19}. This was obtained via the classical tool in Thermodynamic Formalism known as Ruelle-Perron-Frobenius Theorem. It was assumed that the kernel of the bounded linear operator $L$ is nontrivial and of finite dimension. In order to define this operator, it was necessary in  \cite{LMSV19} to fix an {\em a priori} probability on the kernel of $T$  with some properties.

For most of the reasoning of the present paper. we consider similar assumptions as in \cite{LMSV19}. We introduce the concepts of entropy and pressure in the setting of Linear Dynamics on Banach spaces. Our first  goal is to show that for a given  potential $A: X \to \mathbb{R}$, satisfying some mild regularity  assumptions, the Gibbs state $\mu_A$ obtained in Theorem $1$ in \cite{LMSV19} is in fact  an equilibrium state associated with $A$ i.e. it is satisfied a variational principle of  pressure and the supremum is attained at $\mu_A$.

Our second goal is to present results on the topic of Ergodic Optimization. We prove that the accumulation points of the family of probability measures $(\mu_{tA})_{t > 1}$, when $t$ goes to $\infty$, are maximizing probability measures for $A$. We point out the existence of accumulation points is not a trivial matter. This will require assuming some properties for the potential $A$ in order to  be able to show that a certain sequence of probabilities is tight.  For some of our results, we will need some more strong assumptions on  the {\em a priori} probability measure than the ones considered in \cite{LMSV19}. Moreover, we show the existence of calibrated sub-actions when the potential is at least H\"older continuous. Our most important new assumption  is to suppose  that the Banach space $X$ has a Schauder basis which is a property satisfied by a wide class of Banach spaces, such as, $c_0(\mathbb{R})$, $l^p(\mathbb{R})$, $1 \leq p < +\infty$ and any  separable Hilbert space  $H$. Moreover, we present examples concerning frequently hypercyclic and Devaney chaotic operators $L : X \to X$, which are  defined on the Banach spaces $X = c_0(\mathbb{R})$ or $X = l^p(\mathbb{R})$, $1 \leq p < +\infty$. This class of operators, which were also considered in  \cite{LMSV19},  are given by the equation
$$
L((x_n)_{n \geq 1}) = (\alpha_n x_{n+1})_{n \geq 1} \;,
$$
where $(\alpha_n)_{n \geq 1}$ is a sequence of real numbers satisfying suitable conditions (see for instance \cite{MR2533318, MR3759568, BeMe20, MR2919812}). This type of operator is known in the classical mathematical literature as weighted shift.

At the end of the paper, we introduce the concept of Ma\~n\'e potential in the setting of Linear Dynamics. Taking advantage of  the Ma\~n\'e potential we prove the  existence of calibrated  sub-actions associated with H\"older continuous potentials, under the assumption that the set of maximizing measures is non-empty. Furthermore, we show that it is not necessary to assume the hypothesis of summable variations on the potential in order to guarantee the result. In the last section  we also present some examples where is possible to get an explicit expression for the calibrated sub-action and is guaranteed the selection at zero temperature.

The paper is organized as follows:

In section \ref{main-results-section} we introduce basic definitions, the notation that will be used throughout the paper and we state the main results to be obtained.

In section \ref{ground-states-section} we prove existence of equilibrium states and ground states in the matter of Linear Dynamics on Banach spaces. More precisely, in section \ref{variational-principle-section} appears the proof of Theorem \ref{variational-principle}  which guarantees that any Gibbs state is, in fact, an equilibrium state. In section \ref{zero-temperature-limit-section} appears the proof of Theorem \ref{limit-theorem}, which guarantees existence of ground states. In section \ref{maximizing-measures-section} we prove existence of maximizing measures via ground states and maximizing potentials in order to prove that the so-called Ma\~n\'e potential is well defined in our matter. Finally, in section \ref{example-section} we present particular cases of our two first main results of the paper in the setting of weighted shifts defined on $X = c_0(\mathbb{R})$ and $X = l^p(\mathbb{R})$, $1 \leq p < +\infty$.

In section \ref{sub-actions-section} we show existence of calibrated sub-actions using the Ma\~n\'e potential and some explicit examples about existence of sub-actions and selection at zero temperature. More precisely, in section \ref{sub-actions-subsection} appears the proof of Theorem \ref{sub-action-theorem} and in section \ref{examples} we present examples to illustrate the theory studied in the paper.

The proofs of some of our results (in the setting Linear Dynamics) are similar to the analogous ones for the shift, in that case, we will just briefly mention to the reader references for the proof. For some other results, the proofs are quite different and we provide full details in this case.

\section{Main results}
\label{main-results-section}

In this section we present  some basic definitions that are required for our reasoning  and we also state the main results of the paper. Consider a Banach space $X$ equipped with a norm $\| \cdot \|_X$. 
Hereafter,  we fix a linear operator $T:X \to X$ which is surjective (but not injective), and we assume that the Banach space $X$ can be decomposed as the direct sum
\[
X = \mathrm{Ker}(T) \oplus E \;,
\]
where $\mathrm{Ker}(T)$  denotes the {\em kernel} of the operator $T$. We also assume throughout the paper that $0 < \dim(\mathrm{Ker}(T)) < +\infty$.

Then, defining
\begin{equation}
\label{p-norm}
p(T) := \inf\{ \|T(x)\|_X : \|x\|_X = 1,\, x \in E \} \,,
\end{equation}
it follows from the open map theorem that $p(T) > 0$ and $p(T)\|x\|_X \leq \|T(x)\|_X$ for each $x \in E$. 

Furthermore, under the assumption that for each $n \in \mathbb{N}$ the Banach space $X$ admits a decomposition of the form
\begin{equation}
\label{n-sum-decomposition}
X = \mathrm{Ker}(T^n) \oplus E_n \;,
\end{equation}
where $T(E_{n+1}) \subset E_n$, it follows that $p(T^n) > 0$. Besides that, as a consequence of the sub-additivity of the sequence $(p(T^n))_{n \geq 1}$, it follows the existence of the limit $\lim_{n \to +\infty}(p(T^n))^{\frac{1}{n}}$  (in particular, the above conditions are satisfied in the case of weighted shifts, when $\alpha_n \in (c', c)$ for some $0 < c' < c < +\infty$).

It is known that the convergence of the series $\sum_{n = 1}^{+\infty}p(T^n)^{-\alpha}$ to a real number implies that the dynamical system $T : X \to X$ is Devaney chaotic and frequently hypercyclic  (see for details \cite{MR2533318,MR2919812}). In particular, the above implies that the dynamical system $T$ is topologically transitive under the assumption that $X$ is a separable Banach space.

For each $x \in X \setminus \{0\}$ and any $n \in \mathbb{N}$, we use the following notation
\[
T^{-n}(x) := T^{-n}(\{x\}) = \{v \in X : T^n(v) = x\} \,.
\]

As $T$ is surjective but not injective, we obtain that $T^{-1}(x)$ is a non-empty and non-singleton set. Moreover, for each $v \in T^{-1}(x)$ it is satisfied  the expression
\[
T^{-1}(x) = \mathrm{Ker}(T) + \{v\} = \{z + v : z \in \mathrm{Ker}(T) \}\,.
\]

In other words, $T^{-1}(x)$ is isometrically isomorphic to
$\mathrm{Ker}(T)$; this property will be quite useful in the  definition of the Ruelle operator (in the same way as  in \cite{LMSV19}).

Denote by $\mathcal{C}(X)$ the set of {\em continuous functions} from $X$ into $\mathbb{R}$ and by $\mathcal{C}_b(X)$ the set of {\em bounded continuous functions} from $X$ into $\mathbb{R}$.  The set $\mathcal{C}_b(X)$ equipped with the uniform norm $\|\cdot\|_\infty$, which is given by $\|\varphi\|_\infty := \sup\{|\varphi(x)| : x \in X\}$,  is a Banach space.

\begin{definition}
 We say that a potential $A \in \mathcal{C}(X)$ has {\em summable variations} with respect to the bounded linear operator $T : X \to X$, if
\begin{equation}
\label{summable-variation}
V_T(A) := \sum_{n = 1}^{+\infty} V_{T,n}(A) < +\infty \;,
\end{equation}
where
\begin{equation}
\label{summable-variation1}
V_{T,n}(A) := \sup\{|A(z_n + x_n) - A(z_n + y_n)| :\; z_n \in \mathrm{Ker}(T^n),\; x_n, y_n \in E_n\}.
\end{equation}
\end{definition}

We denote by $\mathcal{SV}_T(X)$ the set of potentials with summable variations with respect to $T$, i.e., the ones satisfying the equation in \eqref{summable-variation}. Note that the {\it summable variations property} doesn't implies boundedness of the potential $A$. However, this property guarantees a good behavior of the images of the Ruelle operator associated to the potential $A$ when the operator acts on the set of bounded continuous functions.

Given a function $\varphi \in \mathcal{C}(X)$ and $\alpha \in (0, 1]$, define
\[
\mathrm{Hol}^{\alpha}_{\varphi} := \sup_{x \neq y}\frac{|\varphi(x) - \varphi(y)|}{\|x - y\|^{\alpha}_X} \;.
\]

A function $\varphi$ is called {\em $\alpha$-H\"older continuous}, if $\mathrm{Hol}^{\alpha}_{\varphi} < +\infty$. The set of all the $\alpha$-H\"older continuous functions from $X$ into $\mathbb{R}$ is denoted by $\mathcal{H}_{\alpha}(X)$ and the set of all the {\em bounded $\alpha$-H\"older continuous functions} is denoted by $\mathcal{H}_{b, \alpha}(X)$.

\smallskip

Given  $T$ as above, we consider an {\em a priori} probability measure $\nu$ on the kernel of $T$ whose support is all the kernel. For instance, when the kernel has dimension $n$, we could take $\nu$ as the Gaussian distribution on $\mathbb{R}^n$ (which is isomorphic to the kernel of $T$) with mean zero and variance $1$. More precisely, taking $\nu := f\;dx$, with $f(x) := \frac{1}{(2\pi)^{n/2}}e^{-\|x\|_2^2/2}$.

\begin{definition}  \label{Def:adapted}
We say that a probability $\nu$ has {\em strong adapted tails} if, for any $\epsilon > 0$ there exists a sequence of positive numbers $(\kappa_n)_{n \geq 1}$ such that
\begin{enumerate}
\item $\sum_{n=1}^{+\infty} \nu([-p(T^n) \kappa_n,\; p(T^n) \kappa_n]^{c}) < \epsilon$;
\item the sequence $ (\kappa_n)_{n \geq 1} $ belongs to $l^1(\mathbb{R})$.
\end{enumerate}
\end{definition}

The above conditions are stronger than the ones presented  in the definition of {\it adapted tails measure} that appears in \cite{LMSV19}. That is, any probability measure satisfying the conditions of Definition \ref{Def:adapted} is also  a measure with adapted tails in the sense of \cite{LMSV19}.
\smallskip

Now, we define a Ruelle operator acting on the set $\mathcal{C}_b(X)$, which will help us to find  the so-called  equilibrium states and ground states, via the existence of Gibbs states and the variational principle. Given a bounded above potential $A \in \mathcal{H}_\alpha(X)$ and an {\em a priori} Borelian probability measure $\nu$, supported on $\mathrm{Ker}(T)$ and satisfying the strong adapted tails property, the {\em Ruelle operator} $\mathcal{L}_A$ of the potential $A:X \to \mathbb{R}$ is defined as the map assigning to each $\varphi \in \mathcal{C}_b(X)$ the function $\mathcal{L}_A(\varphi) \in \mathcal{C}_b(X)$ given by
\begin{equation}
\label{Ruelle-operator}
\mathcal{L}_A(\varphi)(x) := \int_{z \in \mathrm{Ker}(T)}e^{A(z + v)}\varphi(z + v)d\nu(z), \; \; v \in T^{-1}(x) \;.
\end{equation}

We say that the potential $A$ is {\it normalized} (for the {\it a priori} probability $\nu$) if  $\mathcal{L}_A(1)=1.$ 

Denote by $\mathcal{B}(X)$ the set of {\em finite Borelian measures} on $X$. By well-known  properties of the dual operator (see for instance \cite{LMSV19}), we can define the  dual Ruelle operator $\mathcal{L}_A^*$   as the map that assigns to each $\mu \in \mathcal{B}(X)$ the finite Borelian measure $\mathcal{L}_A^*(\mu)$, which for each $\varphi \in \mathcal{C}_b(X)$ satisfies
\[
\int_X \varphi d\Bigl( \mathcal{L}_A^*(\mu) \Bigr) := \int_X \mathcal{L}_A(\varphi) d\mu \;.
\]

We follow the basic lines of the proof of the main Theorem in \cite{LMSV19}, with a modification in the part that guarantees the existence of the main eigenfunction. We adapt the reasoning of  Section $3$ in \cite{CSS19}, to the linear dynamics setting (see also \cite{BS16}). In this case it is not difficult to check that for any bounded above potential $A \in \mathcal{H}_\alpha(X) \cap \mathcal{SV}_T(X)$, there is an eigenvalue $\lambda_A > 0$ and a strictly positive eigenfunction $\psi_A \in \mathcal{H}_{b, \alpha}(X)$ associated with $\lambda_A$. Moreover, under the assumption that $\mathcal{L}_A(1) = 1$, it is possible to guarantees the existence of a {\em Gibbs state} $\mu_A$ (i.e. a fixed point of the operator $\mathcal{L}^*_A$).

The hypothesis on \cite{LMSV19} was that the potential was bounded (above and below) and we will need  here for some of our results a lack of  boundedness by below of the potential $A$. 


Note that even in the case when $A$ is not a normalized potential, we are able to guarantee the existence of an eigenprobability $\rho_A$ (for the dual Ruelle operator $\mathcal{L}_A^*$) associated with the eigenvalue $\lambda_A$. Actually, the probability  $\rho_A$ has the same support of the Gibbs state associated with the associated normalized potential of the potential $A$ (to  be defined below). Indeed, since the potential
$$
\overline{A} := A + \log(\psi_A) - \log(\psi_A \circ T) - \log(\lambda_A) \;.
$$
is a normalized potential, one gets that $\mathcal{L}_{\overline{A}}^*\mu_{\overline{A}} = \mu_{\overline{A}}$. Thus, denoting $\mu_A = \mu_{\overline{A}}$, it follows that $\rho_A = \frac{1}{\psi_A} d\mu_A$ satisfies the desired properties.   

We call $\overline{A}$ the {\em associated normalized potential} for $A$, we say that $\mu_A$ is the {\em Gibbs state} for the (non-normalized) potential $A$, $\rho_A$ is the {\em conformal measure} associated with $A$ and $\lambda_A > 0$ is the {\em main eigenvalue} of $\mathcal{L}_A$.

Hereafter, we use the notation $\mathcal{P}(X)$ for the set of {\em Borelian probability measures} on $X$ and we denote by $\mathcal{P}_T(X) \subset \mathcal{P}(X)$ the ones that are {\em invariant} by the action of $T$.

\begin{definition}
Given an {\em a priori} probability measure $\nu$,  supported on the Kernel of $T$, and $\mu \in \mathcal{P}_T(X)$, the {\em entropy} of $\mu$ is defined as
\begin{equation} \label{rrt}
h_\nu (\mu) := \inf\Bigl\{ \int_X \log\Bigl(\frac{\mathcal{L}_0(u)}{u}\Bigr) d\mu : u \in \mathcal{C}_b(X),\; u > 0 \Bigr\}.
\end{equation}
\end{definition}

\begin{remark} \label{neg}
The values of $h_\nu (\mu)$ are non-positive and for the measure $\mu$ of maximal entropy the value is $h_\nu (\mu)=0$.  In this last case, we can guarantee the existence of such $\mu$, because the potential $B \equiv 0$ is a normalized potential.
\end{remark}

Considering the symbolic space $\{1, 2, ... , d\}^\mathbb{N}$, the shift map $\sigma: \{1, 2, ... , d\}^\mathbb{N} \to \{1, 2, ... , d\}^\mathbb{N},$ and taking $\nu$ as the {\it counting measure} on $\{1, 2, ... , d\}$,  the real positive value obtained from the analogous expression to \eqref{rrt} (using the classical Ruelle operator), is exactly the Kolmogorov-Shannon entropy (see \cite{Lo1,MR3377291,SV20}). If we take $\nu$ as the {\it normalized counting probability} on $\{1, 2, ... , d\}$, then, we get that the value obtained from \eqref{rrt} is the Kolmogorov-Shannon entropy minus the value $\log d$ (therefore, a non  positive number as explained in \cite{MR3377291}).

When the set of preimages of any point with respect to the dynamics is not a countable set, it is not appropriate to define entropy via dynamical partitions. This happens for instance when considering the shift acting on the symbolic space $M^\mathbb{N}$, where $M$ is a compact metric space (like the case where $M$ is the unitary circle or the interval $[0, 1]$). Then, alternatively, one can define entropy  via the information provided by the Ruelle operator (which depends on an {\em a priori} probability measure $\nu$ as above in \eqref{rrt}, and taking the Ruelle operator associated with the potential which is constant equal to zero). We point out that (in the general case)  it is required to take  $\nu$ as a probability (and not a measure) in order for the expression  \eqref{rrt} to be well-defined .  By taking $\nu$ as a probability (not a infinite measure) the value we obtain
in \eqref{rrt} is non-positive, and the maximal possible value of the entropy of an invariant probability is the value zero. We refer the reader to  \cite{MR3377291} and \cite{SV20} for more details.

In this way,  it is natural to adapt this point of view in our setting; defining entropy of a $T$-invariant probability via (\ref{rrt}). It is important to notice that (\ref{rrt}) depends on the so-called  {\em a priori} probability measure. Hence, fixing a probability (for instance the natural standard Gaussian measure on $\mathrm{Ker}(T)$), such definition results in a topological invariant with respect to the map $T$.

\begin{definition}
Given a bounded above potential $A \in \mathcal{H}_\alpha(X) \cap \mathcal{SV}_T(X)$, we call the  {\em pressure} of $A$  the value
\[
P_\nu(A) := \sup\Bigl\{ h_\nu(\mu) + \int_X A d\mu : \mu \in \mathcal{P}_T(X) \Bigr\} \,.
\]

In addition, we say that $\widehat{\mu} \in  \mathcal{P}_T(X) $ is an {\em equilibrium state} for $A$, if
$$
P_\nu(A)= h_\nu(\widehat{\mu}) + \int_X A d\widehat{\mu} \;.
$$
\end{definition}

The first result of our paper claims  that given a potential $A$ satisfying suitable conditions (see \cite{LMSV19} for the assumptions),  the set of Gibbs states associated with $A$ is contained into the set of equilibrium states for $A$ and $P_\nu(A) = \log(\lambda_A)$, where $\lambda_A$ is the main eigenvalue of the Ruelle operator $\mathcal{L}_A$. Actually, in \cite{LMSV19} the authors assume that the potential $A$ defining the Ruelle operator is bounded and H\"older continuous. However, as we already mentioned,  the proof in \cite{LMSV19} also guarantees the existence of such Gibbs states when the potential $A$ is H\"older continuous and bounded above.

The precise statement of our first result  is the following:

\begin{theorem}
\label{variational-principle}
Consider a separable Banach space $X$ and a bounded linear operator $T : X \to X$, such that, $T$ is surjective but not bijective. Assume also that $X = \mathrm{Ker}(T^n) \oplus E_n$,  with $T(E_{n+1}) \subset E_n$ for each $n \in \mathbb{N}$, $0 < \dim(\mathrm{Ker}(T)) < +\infty$, and, moreover, suppose that $\sum_{n=1}^{+\infty} p(T^n)^{-\alpha} < +\infty$. For each bounded above potential $A \in \mathcal{H}_\alpha(X) \cap \mathcal{SV}_T(X)$ denote the Gibbs state associated with $A$ by $\mu_A$. Then, the following variational principle is satisfied
\[
P_\nu(A) = \log(\lambda_A) = h_\nu(\mu_A) + \int_X A d\mu_A \,.
\]

That is, the Gibbs state $\mu_A$ associated with $A$ is also an equilibrium state for the potential $A$.
\end{theorem}

Theorem \ref{variational-principle} is widely known in the mathematical literature as the variational principle of pressure, or  the Ruelle Theorem. This result will be a quite useful instrument throughout the paper.  We will use it as a tool to find the so-called  {\it maximizing measures} and {\it ground states} of the potential $A$ in the setting of Linear Dynamics. Below we establish conditions in order to guarantee the existence of such measures. We would like to point out to the reader that the proof of this variational principle depends exclusively on properties of the Ruelle operator defined in \eqref{Ruelle-operator}. However, we include it in order to facilitate the understanding of the theory in the setting  that we are interested in.

\begin{definition}
A {\em Schauder basis} for the Banach space $X$ is a  sequence $(e_k)_{k \geq 1}$ of vectors in $X$, such that, for each $x \in X$ there is a unique sequence of real numbers $(\alpha_k)_{k \geq 1}$ satisfying
\begin{equation}
\label{Schauder-basis}
\lim_{n \to +\infty} \Bigl\|x - \sum_{k=1}^{n} \alpha_k e_k\Bigl\|_X = 0 \,.
\end{equation}
\end{definition}

First note that any Banach space equipped with a Schauder basis results in a separable space. Furthermore, it is widely known that any Schauder basis for $X$ induces a corresponding basis of {\em coordinate functions} for the dual space $X'$, which is given by a sequence $(\pi_k)_{k \geq 1}$ on $X'$, such that, for each $i, j \in \mathbb{N}$ it is satisfied $\pi_j(e_i) = \delta_{ij}$. Hereafter, we use the notation $x := \sum_{k=1}^{+\infty} \alpha_k e_k$ when the vector $x \in X$ satisfies the limit in \eqref{Schauder-basis} for the values $\alpha_k = \pi_k(x)$ for each $k \in \mathbb{N}$. Actually, note that the sequence $(\pi_k)_{k \geq 1}$ is a {\em total subset} of $X'$, i.e., $\pi_k(x) = 0$ for each $k \in \mathbb{N}$ implies that $x = 0$.

There are several examples of Banach spaces equipped with a Schauder basis, for instance $X \in \{c_0(\mathbb{R}),\; l^p(\mathbb{R}),\; 1 \leq p < +\infty\}$ and $X = H$, where $H$ is an arbitrary separable Hilbert space; these are typical examples of spaces satisfying that property. In particular, the set of coordinate functions agrees with the Schauder basis itself when $X = H$ is an arbitrary separable Hilbert space. Moreover, in the last case, the Schauder basis is called as {\em Hilbert basis} of $H$.

\begin{definition}
\label{summability-definition}
Consider a Banach space $X$ equipped with a Schauder basis $(e_k)_{k \geq 1}$. Define the set
\[
X_{i, j} := \{x \in X :\; j \leq |\alpha_i| \leq j + 1\}
\]

We say that a potential $A \in \mathcal{C}(X)$ satisfies the {\em summability condition}, if for each $i \in \mathbb{N}$ is satisfied
\begin{equation}
\label{summability-condition}
\sum_{j = 1}^{+\infty} e^{\sup\{A(x) :\; x \in X_{i, j} \}} < +\infty \,.
\end{equation}
\end{definition}

In particular, the so-called  {\it summability condition} implies that for each $i \in \mathbb{N}$ it is satisfied the property $\lim_{j \to +\infty} \sup\{A(x) :\; x \in X_{i, j} \} = -\infty$. Hence, in that case we get $\sup(A) < +\infty$, but $A$ cannot be a bounded below potential.

\begin{definition}
Given a potential $A \in \mathcal{C}(X)$, a measure $\mu_{\max} \in \mathcal{P}_T(X)$ is called a {\em maximizing measure} for $A$, if $\int_X A d \mu_{\max} = m(A)$, where
\begin{equation}
\label{maximizing-measure}
m(A) := \sup\Bigl\{\int_X A d\mu : \mu \in \mathcal{P}_T(X) \Bigr\} \;.
\end{equation}

We denote by $\mathcal{P}_{\max}(A)$ the set of invariant probability measures attaining the supremum in \eqref{maximizing-measure}, which is usually called as the set of {\em maximizing measures} of $A$.
\end{definition}

It is important to point out that for any potential $A$ in $\mathcal{C}(X)$ satisfying the summability condition, we have $m(A) < +\infty$. However, even assuming that condition, it is possible that $\mathcal{P}_{\max}(A) = \emptyset$. In the following example, we present a case where the set of maximizing measures is a non-empty subset of $\mathcal{P}_T(X)$.

\begin{example} \label{vio}
Consider a bounded above potential $A : X \to \mathbb{R}$. When $v$ is such that $T^k(v) = v$ the probability measure $\widetilde{\mu}$ with support in the periodic orbit
$\{v, T(v), T^2(v), ... ,T^{k-1}(v)\}$ of the form
$$
\widetilde{\mu} := \frac{1}{k} \,\sum_{j=0}^{k-1} \delta_{T^j(v)} \;
$$
is $T$-invariant. Then, in this case $\frac{1}{k} \sum_{j=0}^{k-1} A (T^j(v))\leq  m(A)$. If for all $j = 0, 1, 2, ... , k-1$, we have that  $A(T^j(v))= m(A)$, it follows that $\widetilde{\mu}$ is maximizing for $A$
\end{example}

In order to prove the existence  of ground states and maximizing measures in the case of bounded above potentials belonging to the set $\mathcal{H}_\alpha(X) \cap \mathcal{SV}_T(X)$, we will take advantage of certain properties of  Gibbs states, which are consequences of the main Theorem in \cite{LMSV19}. The next theorem claims that under suitable assumptions for the potential $A$, the set of ground states is a non-empty set. Note that this  is a non-trivial claim. This is so, by the lack of compactness of the Banach space $X$ and also by the fact that the closed balls are not compact sets on infinite-dimensional Banach spaces (by the Theorem of Riesz).

The statement of the result is as follows.

\begin{theorem}
\label{limit-theorem}
Let $X$ be a Banach space equipped with a Schauder basis $(e_k)_{k \geq 1}$ and consider $T : X \to X$ a bounded linear operator surjective but not bijective. Also assume that $\sum_{n=1}^{+\infty} p(T^n)^{-\alpha} < +\infty$ and $\mathrm{Ker}(T^n) = \mathrm{span}\{e_{m_1}, ... , e_{m_n}\}$ for each $n \in \mathbb{N}$, where $(m_i)_{i \geq 1}$ is some bijective sequence of natural numbers. Then, for any potential $A \in \mathcal{H}_\alpha(X) \cap \mathcal{SV}_T(X)$ satisfying the summability condition, the family of equilibrium states $(\mu_{tA})_{t>1}$ has accumulation points at infinity - denoted generically by $\mu_\infty$. Furthermore, the probability $\mu_\infty$ will be maximizing for $A$.
\end{theorem}

 Our second interest in this paper is to be able to locate the union of the supports  of maximizing measures associated with a fixed potential $A \in \mathcal{C}(X)$.     
A tool widely used for this purpose, in the classic literature on Ergodic Optimization, is based on calibrated sub-actions. In order to use this tool, some hypotheses are needed. All of this, of course, for the case where $\mathcal{P}_{\max}(A)$ is a non-empty set.  Note that the last Theorem provides a useful way to guarantee existence of maximizing probabilities via ground states. Below, we present another technique that allows existence of maximizing measures in a more general setting where the potential does not necessarily satisfies decay conditions.

We would like to mention  that some results about the existence  of maximizing measures (where the underlying space is not compact)  for a certain class of potentials, defined on some Polish spaces with bounded expansive metrics, were presented  in \cite{MR2354972}.

Another  class of potentials where  $\mathcal{P}_{\max}(A)$ is a non-empty set is the following:

\begin{definition}
\label{maximizing-property}
Assume that the potential $A$ belongs to $\mathcal{C}(X)$. We say that $A$ satisfies the {\em maximizing property}, if there are vector subspaces $Y, Z \subset X$, such that $X = Y \oplus Z$, $0 < \dim(Y) < +\infty$, and for each $x = x_y + x_z$ belonging to $X$, with $x_y \in Y$ and $x_z \in Z$, we have
\begin{equation}
\label{max-prop-1}
A(x) \leq A(x_y) \;,
\end{equation}
and
\begin{equation}
\label{max-prop-2}
\lim_{\|x_y\|_X \to +\infty} A(x_y) = -\infty \;.
\end{equation}
\end{definition}

Note that the  maximizing property for $A$  implies that $\sup(A) < +\infty$. In Lemma \ref{maximizing-property-lemma} we prove that under suitable conditions of regularity for the potential $A$, the maximizing property implies that  $\mathcal{P}_{\max}(A)$ is a non-empty set.

Actually, in Lemma \ref{maximizing-ground-states-lemma} we  prove that $\mathcal{P}_{\max}(A) \neq \emptyset$ when there exist ground states (the last claim stated in Theorem \ref{limit-theorem}) Given $A$, the property $\mathcal{P}_{\max}(A) \neq \emptyset$ will guarantee that the sub-action obtained via  the Ma\~n\'e potential is well-defined. 
\begin{definition}
Given a bounded above potential $A \in \mathcal{C}(X)$, we say that a function $V \in \mathcal{C}(X)$ is a {\em sub-action} associated with $A$, if satisfies the following inequality
\begin{equation}
\label{sub-action}
V \circ T \geq V + A - m(A) \;.
\end{equation}
\end{definition}

There are continuous potentials $A$ without a sub-action (see for instance \cite{MR2191393}  and \cite{J2}). In general, some regularity of the potential $A$ is required for the existence of sub-actions (see \cite{MR1855838}).

It is easy to see that for any maximizing probability $\mu$ for $A$ and each sub-action $V$ associated with $A$, the support of $\mu$ is contained into the set
\begin{equation}
\label{util}
\Omega(A) := \{x \in X :\; V(T(x)) - V(x) - A(x) + m(A) = 0\} \;.
\end{equation}

It is also true that any invariant probability measure $\mu$ with support contained in the set $\Omega(A)$, results in a maximizing probability for $A$ (see \cite{MR1855838}). In this way, a sub-action could help to find explicitly the support of a maximizing measure (see for instance \cite{MR1855838,MR3114331, MR3701349}).

\begin{definition}
Let $V$ be a sub-action for $A$, if for any $x \in X$ there exists a point $y \in T^{-1}(x)$ which attains the equality in \eqref{sub-action} i.e.
$$
V(x) = V(y) + A(y) - m(A) \;,
$$
we say that $V$ is a {\em calibrated sub-action}.
\end{definition}

\begin{remark}
Note that the definition of calibrated sub-action appearing above is equivalent to say that 
\[
V(x) := \max_{y \in T^{-1}(x)} \{V(y) + A(y) - m(A)\} \;.
\]

The so-called sub-actions are also known in the mathematical literature as revelations (for details see \cite{GL1, MR2191393,J2}).
\end{remark}

There are cases where one can show existence of maximizing probabilities for continuous potentials even if  we cannot apply the Ruelle-Perron-Frobenius Theorem formalism. In this case, we can consider maximizing probabilities without talking about ground states (also known in the mathematical literature as limits at temperature zero). Sub-actions will be helpful anyway (see the expression  in \eqref{util}). Actually, in example \ref{maine} we will consider a potential which  is not bounded below but exhibits an explicit sub-action.

In order to prove the existence  of sub-actions we introduce an additional tool known in the classical mathematical literature as Ma\~n\'e potential (for details see  \cite{MR1855838,GL1} and section 3 in \cite{CLO}).

\begin{definition}
Given a potential $A \in \mathcal{H}_{\alpha}(X)$ we define the {\em Ma\~n\'e potential} $\phi_A$ associated with $A$ as the map given by
\begin{equation}
\label{Mane-potential}
\phi_A(x, y) := \lim_{\epsilon \to 0}\sup_{n \in \mathbb{N}}\sup_{\substack{x' \in T^{-n}(y) \\ \|x - x'\|_X < \epsilon}} S_n(A - m(A))(x') \;,
\end{equation}

where $S_n(A)(x) := \sum_{j=0}^{n-1}A(T^j(x))$ is the $n$-th ergodic sum of the potential $A$  with respect to $T$.
\end{definition}

We claim that under the assumptions in Theorem \ref{sub-action-theorem}, we have that $\phi_A(x, y)$ belongs to $(-\infty, 0]$, for any pair of points $x, y \in X$ (for details of the proof see Section \ref{sub-actions-subsection}).

Now, we will present conditions to guarantee the existence of sub-actions associated with H\"older continuous potentials $A$ using ideas related to  the so-called Ma\~n\'e potential. 

The statement of our third main result is the following:

\begin{theorem}
\label{sub-action-theorem}
Consider $X$ a separable Banach space and $T : X \to X$ a bounded linear operator surjective but not injective. Assume that $X = \mathrm{Ker}(T^n) \oplus E_n$,  with $T(E_{n+1}) \subset E_n$ for each $n \in \mathbb{N}$, $0 < \dim(\mathrm{Ker}(T)) < +\infty$, and suppose that $\sum_{n=1}^{+\infty} p(T^n)^{-\alpha} < +\infty$.  Suppose also that $A \in \mathcal{H}_\alpha(X)$ and $\mathcal{P}_{\max}(A)$ is a non-empty set. Then, for each $x \in X$ the map $y \mapsto \phi_A(x, y)$  belongs to $\mathcal{H}_{\alpha}(X)$ and it is a calibrated sub-action for $A$.
\end{theorem}

An example of a potential $A$ where the hypotheses of the above Theorem are satisfied is $A(x) = -\|x - v\|_X$, where $v$ is a fixed point for $T$. In this case, the probability with support on $v$ is maximizing for $A$ (see Example \ref{maine} for a more complete discussion).

\section{Existence of equilibrium states and maximizing measures}
\label{ground-states-section}

In this section, we present the proof of Theorem \ref{variational-principle}, which asserts the existence  of equilibrium states for the Gibbs states obtained in the Ruelle-Perron-Frobenius Theorem formalism. We also present an example in the context of weighted shifts (see for instance \cite{MR2533318, MR3759568, BeMe20, MR2919812}).  This will be proved in section  \ref{variational-principle-section}.

On other hand, section \ref{zero-temperature-limit-section} is dedicated to another topic. We propose two different roads to guarantee the existence of maximizing measures. The first one, through the existence of the so-called  ground states for potentials satisfying a condition of decay called in this paper the summability condition. The second one, using the so-called  maximizing property, in which we assume, {\it in some way}, that the potential attains its maximum on a finite dimensional vector subspace of $X$. Our purpose is to be able to set conditions such that the  Ma\~n\'e potential is well-defined.

We assume throughout the section that $X$ is a separable Banach space, $T : X \to X$ is a bounded linear operator surjective, but not bijective. We also assume that for each natural number $n$, we get $X = \mathrm{Ker}(T^n) \oplus E_n$,  with $T(E_{n+1}) \subset E_n$, $0 < \dim(\mathrm{Ker}(T)) < +\infty$ and $\sum_{n = 1}^{+\infty}p(T^n)^{-\alpha} < +\infty$.

\subsection{A variational principle}
\label{variational-principle-section}

Theorem \ref{variational-principle} claims existence of an equilibrium state for each $\alpha$-H\"older continuous potential $A$ with summable variations. First note that, by \cite{LMSV19}, it is guaranteed existence of the eigenvalue $\lambda_A > 0$, a strictly positive eigenfunction $\psi_A \in \mathcal{H}_{b, \alpha}(X)$ and a Gibbs state $\mu_A \in \mathcal{P}_T(X)$. The Gibbs state is the natural candidate to be the equilibrium state associated with the potential $A$. The next Lemma is the first step for proving the claim of our first main result.

\begin{lemma}
\label{lemma-variational-principle}
Consider $A \in \mathcal{H}_{\alpha}(X)$ and $T : X \to X$ satisfying the conditions of Theorem \ref{variational-principle}. Then, for any $\mu \in \mathcal{P}_T(X)$ it is satisfied the inequality
\bigskip\noindent
\begin{equation}
\label{vp-equation}
h_\nu(\mu) + \int_X A d\mu \leq \log(\lambda_A) \;.
\end{equation}
\bigskip\noindent
\end{lemma}
\begin{proof}
Assume that $\mu \in \mathcal{P}_T(X)$. Then,
\begin{align*}
h_\nu(\mu) + \int_X A d\mu
&= \inf\Bigl\{ \int_X \log\Bigl(\frac{\mathcal{L}_0(u)}{u}\Bigr)d \mu : u \in \mathcal{C}_b^+(X) \Bigr\} + \int_X A d\mu  \\
&\leq \int_X \log\Bigl(\frac{\mathcal{L}_0(e^A\psi_A)}{\psi_A}\Bigr)d \mu 
= \log(\lambda_A)  \;.
\end{align*}
\end{proof}

In order to prove our first main result, we will  show that the supremum in the expression in the left-hand of \eqref{vp-equation} is attained by the so-called  Gibbs state obtained in Theorem $1$ of \cite{LMSV19}.

\begin{proof}[Proof of Theorem \ref{variational-principle}]
By Lemma \ref{lemma-variational-principle}, taking supremum on the set of all the $T$-invariant measures, it follows that
\begin{equation}
\label{vp-1}
P_\nu(A) = \sup\Bigl\{ h_\nu(\mu) + \int_X A d\mu : \mu \in \mathcal{P}_T(X) \Bigr\} \leq \log(\lambda_A) \;.
\end{equation}

On other hand, choosing $u_A = e^{\overline{A}}$, it follows that
\[
\int_X \log\Bigl(\frac{\mathcal{L}_0(u_A)}{u_A}\Bigr) d\mu_A = -\int_X \overline{A} d\mu_A \;.
\]

Furthermore, any $u \in \mathcal{C}_b^+(X)$ can be expressed as  $u = ve^{\overline{A}}$, for some $v \in \mathcal{C}_b^+(X)$, which implies that
\[
\int_X \log\Bigl(\frac{\mathcal{L}_0(u)}{u}\Bigr) d\mu = \int_X \log\Bigl(\frac{\mathcal{L}_{\overline{A}}(v)}{v}\Bigr) d\mu_A - \int_X \overline{A} d\mu_A \geq - \int_X \overline{A} d\mu_A \;.
\]

That is,
\begin{equation}
\label{vp-2}
h_\nu(\mu_A) = - \int_X \overline{A} d\mu_A = - \int_X A d\mu_A + \log(\lambda_A) \;.
\end{equation}

Therefore, since $\mu_A \in \mathcal{P}(X)$, by \eqref{vp-1} and \eqref{vp-2}, it follows that
\[
h_\nu(\mu_A) + \int_X A d\mu_A = \log(\lambda_A) = \sup\Bigl\{ h_\nu(\mu) + \int_X A d\mu : \mu \in \mathcal{P}_T(X) \Bigr\} := P(A) \,.
\]

\end{proof}

\begin{remark}
Given a potential $A$, the above variational principle only requires the existence of a well-defined Ruelle operator $\mathcal{L}_A$ acting on the space of bounded continuous function and the existence of a function
$\psi_A $ satisfying $\mathcal{L}_A(\psi_A) = \lambda_A\psi_A$.
\end{remark}

\subsection{Accumulation points at zero temperature}
\label{zero-temperature-limit-section}

In this section we present the proof of Theorem \ref{limit-theorem} which is the second main result of our paper. Our strategy is to adapt a classical result used for finding accumulation points of families of probability measures defined on Banach spaces satisfying suitable hypothesis. Such result is widely known in the mathematical literature as the Prohorov's method for characteristic functionals (see for instance \cite{Aco70, Ara78, Pro61}). In order to apply this result to our setting, it is necessary to define an adequate family of sets with large enough mass. More specifically, we need to find suitable sets, depending on the set of linear functionals defined on the Banach space, with measure close to one for all the members of the family $(\mu_{tA})_{t > 1}$ (this implies the so-called tightness condition). Furthermore, the above guarantee that the family of equilibrium states $(\mu_{tA})_{t > 1}$ has an accumulation point at $\infty$ as a consequence of the Prohorov's Theorem \cite{MR1700749, Pro56}.

Throughout this section we assume that the space $X$ is equipped with a fixed Schauder basis $(e_k)_{k \geq 1}$. That is, for any $x \in X$ there is a unique sequence of real numbers $(\alpha_k)_{k \geq 1}$ satisfying  equation \eqref{Schauder-basis}. Furthermore, we assume  that for each $k \in \mathbb{N}$ it is satisfied the equality $\alpha_k = \pi_k(x)$, where $(\pi_k)_{k \geq 1}$ is the corresponding basis of coordinate functions for the dual space $X'$.

Consider a bijective sequence of natural numbers $(m_i)_{i \geq 1}$. Given $k \in \mathbb{N}$, we define
\begin{equation}
\label{cylindrical-set-i}
\Lambda_k := \{\pi_{m_1}, ... , \pi_{m_k}\} \subset X' \;.
\end{equation}

Now, we define the function $\Pi_{\Lambda_k} : X \to \mathbb{R}^k$ assigning to each point $x \in X$ the value 
\[
\Pi_{\Lambda_k}(x) := (\pi_{m_1}(x), ... , \pi_{m_k}(x)) \;.
\] 

Given any Borel set of the form $B_k = \prod_{i=1}^k [a_i, b_i] \subset \mathbb{R}^k$, where $(a_i)_{i \geq 1}$ is a strictly decreasing sequence of real numbers and $(b_i)_{i \geq 1}$ is a strictly increasing sequence of real numbers with $a_1 < b_1$. We say that $C_{\Lambda_k}(B_k)$ is a cylindrical set generated by $B_k$ and $\Lambda_k$, if it is of the form
\begin{equation}
\label{cylindrical-set-ii}
C_{\Lambda_k}(B_k) := \Pi_{\Lambda_k}^{-1}(B_k) \subset X \;.
\end{equation}

We point out that the definition of cylindrical sets given above is a    particular case of the one that appears on  page  pp. $406$ in \cite{Pro61}. It is easy to check that the set defined by
\[
\bigcap_{i = 1}^k \{x \in X :\; a_i \leq \pi_{m_i}(x) \leq b_i\},
\]
results in a cylindrical set. The above, because
\[
\bigcap_{i = 1}^k \{x \in X :\; a_i \leq \pi_{m_i}(x) \leq b_i\} = \Pi_{\Lambda_k}^{-1}\Bigl( \prod_{i = 1}^k [a_i, b_i] \Bigr).
\]

Since $(\pi_k)_{k \geq 1}$ is a total subset of $X'$, by Lemma $3$ in \cite{Pro61}  (see also \cite{Aco70} and \cite{Ara78} for a probabilistic approach), we can guarantee that a family of Borel probability measures $(\mu_s)_{s \in S}$ on $X$ results in a tight family if, and only if, for each $k \in \mathbb{N}$ and any $\epsilon > 0$ there exists a compact set $\mathbb{K}_{k, \epsilon} \subset X$ and a finite family of linear functionals $\Lambda_k$ (which we choose as defined by \eqref{cylindrical-set-i}), such that, for each $s \in S$ it is satisfied the inequality
\begin{equation}
\label{tight-cylindrical-sets}
\mu_s(\Pi_{\Lambda_k}^{-1}(\Pi_{\Lambda_k}(\mathbb{K}_{k, \epsilon}))) > 1 - \epsilon \;.
\end{equation}

In the following Lemma, we show that any cylindrical set  of the form that appears in \eqref{cylindrical-set-ii} satisfies a kind of upper Gibbs inequality. This will be  one of the main tools to prove the existence of accumulation points for the family of equilibrium states. 

\begin{lemma}
\label{Gibbs-property}
Consider a set $\Lambda_k$ as defined in \eqref{cylindrical-set-i} and bounded above potential $A \in \mathcal{H}_\alpha(X) \cap \mathcal{SV}_T(X)$. Also assume that $\mathrm{Ker}(T^n) = \mathrm{span}\{e_{m_1}, ... , e_{m_n}\}$ for all $n \in \mathbb{N}$. Then, for each $\widetilde{x}$ belonging to the cylindrical set $C_{\Lambda_k}(B_k)$ satisfying \eqref{cylindrical-set-ii}, we have the following inequality
\begin{equation}
\label{Gibbs}
\frac{\mu_A(C_{\Lambda_k}(B_k))}{e^{S_k A(\widetilde{x}) - k\log(\lambda_A)}} \leq C \;, 
\end{equation}
where the constant $C$ is given by $C = e^{(1 + 3\kappa_A)V_T(A) + \mathrm{Hol}^{\alpha}_A \mathrm{diam}(B_k) \sum_{i=1}^k \|T\|_{op}^{i-1}}$, with $\| \cdot \|_{op}$ the operator norm.
\end{lemma}
\begin{proof}

In order to simplify our reasoning, we assume first that $A$ is a normalized potential. The above assumption guarantees that $\mathcal{L}_A(1) = 1$ and $\mathcal{L}^*_A(\mu_A) = \mu_A$. Since $a_1 < b_1$, it follows that $a_i < b_i$ for each $i \in \mathbb{N}$. Using the above, given $\delta > 0$ such that $a_1 + \delta < b_1 - \delta$, we are able to define the set $B_k^\delta := \prod_{i=1}^k [a_i + \delta, b_i - \delta]$. 

On other hand, since $C_{\Lambda_k}(B_k^\delta)$ and $C_{\Lambda_k}(B_k)$ are a closed subsets of $X$, there exists a function $\phi \in \mathcal{C}(X)$, such that, 
\[
{\bf 1}_{C_{\Lambda_k}(B_k^\delta)} \leq \phi \leq {\bf 1}_{C_{\Lambda_k}(B_k)} \;.
\]

By \eqref{n-sum-decomposition} (which is satisfied because $\mathrm{Ker}(T^n) = \mathrm{span}\{e_{m_1}, ... , e_{m_n}\}$), it follows that for any $\widetilde{x} \in C_{\Lambda_k}(B_k)$ we have a decomposition $\widetilde{x} = \widetilde{z}_k + \widetilde{v}_k$, with $\widetilde{z}_k \in \mathrm{Ker}(T^k)$ and $\widetilde{v}_k \in E_k$. Furthermore, since $T(E_{n+1}) \subset E_n$ for each natural number $n$, we obtain that $T^i(\widetilde{x}_k) \in E_{k-i}$ and $T^i(\widetilde{z}_k) \in \mathrm{Ker}(T^{k-i})$ for each $i \in \{1, ... , k-1\}$. 

Then, for any $x \in X$ and each $z_k + v_k \in T^{-k}(x) \cap C_{\Lambda_k}(B_k)$, with $z_k \in \mathrm{Ker}(T^k)$ and $v_k \in E_k$, the following is satisfied
\begin{align*}
&S_k A(z_k + v_k) \\ 
&= S_k A(z_k + v_k) - S_k A(\widetilde{z}_k + v_k) + S_k A(\widetilde{z}_k + v_k) - S_k A(\widetilde{x}) + S_k A(\widetilde{x}) \\
&\leq S_k A(\widetilde{x}) + |S_k A(\widetilde{z}_k + v_k) - S_k A(\widetilde{z}_k + \widetilde{v}_k)| + |S_k A(z_k + v_k) - S_k A(\widetilde{z}_k + v_k)| \\ 
&\leq S_k A(\widetilde{x}) + \sum_{i=1}^k V_{T, i}(A) + \mathrm{Hol}^{\alpha}_A \Bigl(\sum_{i=1}^k \|T\|_{op}^{i-1} \|z_k - \widetilde{z}_k\|_X \Bigr) \\
&\leq S_k A(\widetilde{x}) + \sum_{i=1}^k V_{T, i}(A) + \mathrm{Hol}^{\alpha}_A \mathrm{diam}(B_k) \sum_{i=1}^k \|T\|_{op}^{i-1} \\
&\leq S_k A(\widetilde{x}) + V_T(A) + \mathrm{Hol}^{\alpha}_A \mathrm{diam}(B_k) \sum_{i=1}^k \|T\|_{op}^{i-1} \;,
\end{align*}
where $V_T (A)$ is given by \eqref{summable-variation} and $V_{T, i}(A)$ is given by \eqref{summable-variation1}. By the above, it follows that
\begin{align*}
&\mu_A(C_{\Lambda_k}(B_k^\delta))  \\
&= \int_X {\bf 1}_{C_{\Lambda_k}(B_k^\delta)} d\mu_A  \\
&\leq \int_X \phi d\mu_A  \\
&= \int_X \mathcal{L}_A^k(\phi) d\mu_A  \\
&= \int_X \int_{z_k \in \mathrm{Ker}(T^k)} ... \int_{z_1 \in \mathrm{Ker}(T)} e^{S_k A(z_k + v_k)}\phi(z_k + v_k) d\nu(z_1) ... d\nu(z_k) d\mu_A(x)  \\
&\leq \int_X \int_{z_k \in \mathrm{Ker}(T^k)} ... \int_{z_1 \in \mathrm{Ker}(T)} e^{S_k A(z_k + v_k)}{\bf 1}_{C_{\Lambda_k}(B_k)}(z_k + v_k) d\nu(z_1) ... d\nu(z_k) d\mu_A(x)  \\
&\leq e^{S_k A(\widetilde{x}) + V_T(A) + \mathrm{Hol}^{\alpha}_A \mathrm{diam}(B_k) \sum_{i=1}^k \|T\|_{op}^{i-1}} \int_X \int_{\mathrm{Ker}(T^k)} ... \int_{\mathrm{Ker}(T)} d\nu^k d\mu_A  \\
&\leq e^{S_k A(\widetilde{x}) + V_T(A) + \mathrm{Hol}^{\alpha}_A \mathrm{diam}(B_k) \sum_{i=1}^k \|T\|_{op}^{i-1}} \;.
\end{align*}

Therefore, using that the value $\delta > 0$ is arbitrary, we obtain that
\begin{equation}
\label{Gibbs-1}
\frac{\mu_A(C_{\Lambda_k}(B_k))}{e^{S_k A(\widetilde{x})}} \leq e^{V_T(A) + \mathrm{Hol}^{\alpha}_A \mathrm{diam}(B_k) \sum_{i=1}^k \|T\|_{op}^{i-1}}.
\end{equation}

Now, assuming that $A \in \mathcal{H}_\alpha(X) \cap \mathcal{SV}_T(X)$ is not a normalized potential, we consider its corresponding normalization
\[
\overline{A} = A + \log(\psi_A) - \log(\psi_A \circ T) - \log(\lambda_A),
\]
where the potential $\psi_A \in \mathcal{H}_{b, \alpha}(X)$ and the value $\lambda_A > 0$ satisfy the equation $\mathcal{L}_A(\psi_A) = \lambda_A \psi_A$. It is easy to check that $\overline{A} \in \mathcal{H}_\alpha(X) \cap \mathcal{SV}_T(X)$, $\mathcal{L}_{\overline{A}}(1) = 1$ and $\mu_A = \mu_{\overline{A}}$ (see for instance \cite{LMSV19}). 

Note that for any $k \in \mathbb{N}$ and each $\widetilde{x} \in C_{\Lambda_k}(B_k)$ 
\[
S_k \overline{A}(\widetilde{x}) = S_k A(\widetilde{x}) + \log(\psi_A(\widetilde{x})) - \log(\psi_A(T^k(\widetilde{x}))) - k\log(\lambda_A) \;.
\] 

As $\mathcal{L}_A(\psi_A) = \lambda_A \psi_A$, it is not difficult to check that for any $x, y \in X$ it is satisfied the inequality
$$
|\log(\psi_A)(x) - \log(\psi_A)(y)| \leq \kappa_A V_T(A) \;,
$$ 
where $\kappa_A > 0$ is a constant depending exclusively of the potential $A$. The above implies both, that $\sup(\log(\psi_A)) - \inf(\log(\psi_A)) \leq \kappa_A V_T(A)$ and that $V_T(\overline{A}) \leq (1 + 2\kappa_A)V_T(A)$. Thus, from \eqref{Gibbs-1}, it follows immediately that
\begin{align*}
\frac{\mu_A(C_{\Lambda_k}(B_k))}{e^{S_k A(\widetilde{x}) - k\log(\lambda_A)}} 
&\leq e^{V_T(\overline{A}) + \log(\psi_A)(\widetilde{x}) - \log(\psi_A)(T^k(\widetilde{x})) + \mathrm{Hol}^{\alpha}_A \mathrm{diam}(B_k) \sum_{i=1}^k \|T\|_{op}^{i-1}} \\
&\leq e^{V_T(\overline{A}) + \kappa_A V_T(A) + \mathrm{Hol}^{\alpha}_A \mathrm{diam}(B_k) \sum_{i=1}^k \|T\|_{op}^{i-1}} \\
&\leq e^{(1 + 3\kappa_A)V_T(A) + \mathrm{Hol}^{\alpha}_A \mathrm{diam}(B_k) \sum_{i=1}^k \|T\|_{op}^{i-1}} \;.
\end{align*}

Hence, choosing $C = e^{(1 + 3\kappa_A)V_T(A) + \mathrm{Hol}^{\alpha}_A \mathrm{diam}(B_k) \sum_{i=1}^k \|T\|_{op}^{i-1}}$, we get \eqref{Gibbs}, and the proof is finished.
\end{proof}

\begin{remark}
Note that Lemma \ref{Gibbs-property} guarantees that for any set $X_{m_1, j}$, with $j \in \mathbb{N}$ (of the form described by Definition \ref{summability-definition}), we have
\begin{equation}
\frac{\mu_A(X_{m_1, j})}{e^{A(\widetilde{x}) - \log(\lambda_A)}} \leq e^{(1 + 3\kappa_A)V_T(A) + 2\mathrm{Hol}^{\alpha}_A} \;.
\end{equation}

The above is true  because $X_{i, j} = \pi_i^{-1}([-(j+1), -j] \cup [j, (j+1)])$. Moreover, since there is $l \in \mathbb{N}$ such that $T^l(e_i) = e_{m_1}$, by $T$-invariance of $\mu_A$, we have $\mu_A(X_{i, j}) = \mu_A(X_{m_1, j})$ for all $i, j \in \mathbb{N}$. This inequality will be of great importance in the proof of Theorem \ref{limit-theorem}.
\end{remark}

Take $\Lambda_k$  as in \eqref{cylindrical-set-i}. Now, we are able to present a collection of compact sets $\{\mathbb{K}_{k, \epsilon} : \epsilon \in (0, 1)\}$ satisfying \eqref{tight-cylindrical-sets} for the family of equilibrium states $(\mu_{tA})_{t > 1}$. Given $k \in \mathbb{N}$ and a sequence of natural numbers $(n_i)_{i \geq 1}$ such that for each $l \in \mathbb{N}$ is satisfied $n_{m_l} < n_{m_{l+1}}$. We define the sets
\[
X_{\Lambda_k} := \bigl\{ x  \in X :\; \pi_l(x) = 0 ,\; l \notin \{m_1, ... , m_k\} \bigr\}
\]
and
\begin{equation}
\label{compact-set}
\mathbb{K}_k := \Bigl\{ x \in X :\; |\pi_{m_i}(x)| \leq n_{m_i} ,\; 1 \leq i \leq k \Bigr\} \cap X_{\Lambda_k} \subset X \;.
\end{equation}

Therefore, the set $\mathbb{K}_k$ is a non-empty compact subset of the Banach space $X$. Furthermore, it is not difficult to check that
\[
\mathbb{K}_k = \Pi_{\Lambda_k}^{-1}\Bigl( \prod_{i=1}^k[-n_{m_i}, n_{m_i}] \Bigr) \cap X_{\Lambda_k} \;.
\]

Now, we are able to present the proof of the second main result of this paper.

\begin{proof}[Proof of Theorem \ref{limit-theorem}]
First note that given $k \in \mathbb{N}$, $\Lambda_k = \{\pi_{m_1}, ... , \pi_{m_k}\}$, and $Y \subset X$, we have
\begin{align*}
\Pi_{\Lambda_k}(Y)
&= \pi_{m_1}(Y) \times \cdots \times \pi_{m_k}(Y) \\
&= \pi_{m_1}(Y \cap X_{\Lambda_k}) \times \cdots \times \pi_{m_k}(Y \cap X_{\Lambda_k}) \\
&= \Pi_{\Lambda_k}(Y \cap X_{\Lambda_k}) \;.
\end{align*}

Then, for each $k \in \mathbb{N}$, the set $\mathbb{K}_k$ defined in \eqref{compact-set} satisfies
\begin{align*}
\Pi_{\Lambda_k}^{-1}\Bigl(\Pi_{\Lambda_k}(\mathbb{K}_k)\Bigr)
&= \Pi_{\Lambda_k}^{-1}\Bigl(\Pi_{\Lambda_k}\Bigl( \Pi_{\Lambda_k}^{-1}\Bigl( \prod_{i=1}^k[-n_{m_i}, n_{m_i}] \Bigr) \cap X_{\Lambda_k} \Bigr)\Bigr) \\
&= \Pi_{\Lambda_k}^{-1}\Bigl(\Pi_{\Lambda_k}\Bigl( \Pi_{\Lambda_k}^{-1}\Bigl( \prod_{i=1}^k[-n_{m_i}, n_{m_i}] \Bigr) \Bigr)\Bigr) \\
&= \Pi_{\Lambda_k}^{-1}\Bigl(\prod_{i=1}^k[-n_{m_i}, n_{m_i}] \Bigr) \\
&\supset \bigcap_{i = 1}^{+\infty} \pi_i^{-1}([-n_i, n_i]) \;.
\end{align*}

Therefore, for each $t > 1$, we have
\begin{align*}
\mu_{tA}\Bigl( \Pi_{\Lambda_k}^{-1}\Bigl(\Pi_{\Lambda_k}(\mathbb{K}_k)\Bigr) \Bigr)
&\geq \mu_{tA} \Bigl( \bigcap_{i = 1}^{+\infty} \pi_i^{-1}([-n_i, n_i]) \Bigr) \\
&\geq 1 - \sum_{i = 1}^{+\infty}\mu_{tA}(\{x \in X :\; |\alpha_i| \geq n_i \}) \\
&\geq 1 - \sum_{i = 1}^{+\infty}\mu_{tA}\Bigl(\bigcup_{j = n_i}^{+\infty} \{x \in X :\; j \leq |\alpha_i| \leq j + 1\}\Bigr) \\
&\geq 1 - \sum_{i = 1}^{+\infty}\sum_{j = n_i}^{+\infty}\mu_{tA}(X_{i,j}) \\
&= 1 - \sum_{i = 1}^{+\infty}\sum_{j = n_i}^{+\infty}\mu_{tA}(X_{m_1,j}) \;.
\end{align*}

Thus, in order to define each set $\mathbb{K}_{k, \epsilon}$, with $\epsilon \in (0, 1)$, it is enough to guarantee the existence of a strictly increasing sequence natural numbers $(n^{\epsilon}_i)_{i \geq 1}$, such that,
\[
\sum_{j = n^{\epsilon}_i}^{+\infty}\mu_{tA}(X_{m_1,j}) < \frac{\epsilon}{2^i} \;.
\]

This is so because in this case
\begin{equation}
\label{tightness}
\mu_{tA}\Bigl( \Pi_{\Lambda_k}^{-1}\Bigl(\Pi_{\Lambda_k}(\mathbb{K}_{k, \epsilon})\Bigr) \Bigr) \geq 1 - \sum_{i = 1}^{+\infty}\sum_{j = n^{\epsilon}_i}^{+\infty}\mu_{tA}(X_{m_1,j}) > 1 - \sum_{i = 1}^{+\infty} \frac{\epsilon}{2^i} = 1 - \epsilon \;.
\end{equation}

Consider the potential $B \equiv 0$ which belongs to the set $\mathcal{H}_\alpha(X) \cap \mathcal{SV}_T(X)$. It is easy to check that $\mathcal{L}_B(1) = 1$, and this implies that $\log(\lambda_B) = 0$. Then, by Theorem \ref{variational-principle}, it follows that
\[
0 = h_\nu(\mu_B) \;.
\]

Now, define
\[
I := \int_X A d\mu_B \leq \sup(A) < +\infty \;.
\]

By the above, Theorem \ref{variational-principle} and the fact that $\mu_{tA} = \mu_{t(A - I)}$, it follows that
\begin{align*}
\log(\lambda_{tA}) - tI
&= h_\nu(\mu_{tA}) + t\int_X (A - I) d\mu_{tA}  \\
&= h_\nu(\mu_{t(A - I)}) + t\int_X (A - I) d\mu_{t(A - I)}  \\
&\geq h_\nu(\mu_B) + t\int_X (A - I) d\mu_B  \\
&= h_\nu(\mu_B) = 0  \;.
\end{align*}

Therefore, by \eqref{Gibbs}, for any pair $j \in \mathbb{N}$ and all $x \in X_{m_1,j}$ is satisfied
\begin{align*}
\mu_{tA}(X_{m_1,j})
&\leq e^{tA(x) - \log(\lambda_{tA}) + (1 + 3\kappa_A)V_T(tA) + 2\mathrm{Hol}^{\alpha}_{tA}}  \\
&= e^{t(A(x) - I) - (\log(\lambda_{tA}) - tI) + t(1 + 3\kappa_A)V_T(A) + 2t\mathrm{Hol}^{\alpha}_A}  \\
&\leq e^{t(A(x) - I + (1 + 3\kappa_A)V_T(A) + 2\mathrm{Hol}^{\alpha}_A)}  \\
&\leq e^{t(\sup\{A(x) :\; x \in X_{m_1,j} \} - I + (1 + 3\kappa_A)V_T(A) + 2\mathrm{Hol}^{\alpha}_A)}  \;.
\end{align*}

Now, by \eqref{summability-condition}, it follows that   
\[
\lim_{j \to +\infty} \sup\{A(x) :\; x \in X_{m_1,j} \} = -\infty \;.
\]

The above implies the existence of $j_0 \in \mathbb{N}$, such that, for any $j \geq j_0$ we have
\[
\sup\{A(x) :\; x \in X_{m_1,j} \} - I + (1 + 3\kappa_A)V_T(A) + 2\mathrm{Hol}^{\alpha}_A < 0 \;,
\]
which implies that
\begin{equation}
\label{Gibbs-2}
\mu_{tA}(X_{m_1,j}) \leq e^{\sup\{A(x) :\; x \in X_{m_1,j} \} - I + (1 + 3\kappa_A)V_T(A) + 2\mathrm{Hol}^{\alpha}_A} \;.
\end{equation}

Besides that, also by \eqref{summability-condition}, given $\epsilon > 0$, we can find $n_i^{\epsilon} \geq j_0$, such that
\begin{equation}
\label{Gibbs-3}
\sum_{j = n_i^{\epsilon}}^{+\infty} e^{\sup\{A(x) :\; x \in X_{m_1,j} \}} < \frac{\epsilon}{2^i}e^{I - (1 + 3\kappa_A)V_T(A) - 2\mathrm{Hol}^{\alpha}_A} \,.
\end{equation}

Observe that \eqref{Gibbs-2} and \eqref{Gibbs-3} implies \eqref{tightness}, and from this it follows that for any $t > 1$, we have
\[
\mu_{tA}\Bigl( \Pi_{\Lambda_k}^{-1}\Bigl(\Pi_{\Lambda_k}(\mathbb{K}_{k, \epsilon})\Bigr) \Bigr) > 1 - \epsilon \;.
\]

From the above reasoning it follows  that the family of equilibrium states $(\mu_{tA})_{t > 1}$ is tight (for more  details see Lemma $3$ in \cite{Pro61}). Therefore, it follows from the Prohorov's Theorem the existence of a sequence $(t_n)_{n \geq 1}$, such that, the sequence of equilibrium states $(\mu_{t_n A})_{n \geq 1}$ is convergent, with a limit denoted by  $\mu_\infty$ (see for instance \cite{MR1700749, Pro56}).

It is important to point out that the limits of different subsequences do not have necessarily to be the same.
\end{proof}

\begin{remark}
In the case that $\mathrm{Ker}(T^n) = \mathrm{span}\{e_1, ... , e_n\}$ the condition  \eqref{summability-condition} can be replaced by the condition
\[
\sum_{n = 1}^{+\infty} e^{\sup\{A(x) :\; x \in X_{1,n}\}} < +\infty \,.
\]

Note that in this case, the summability only depends on the behavior of the first component of $x$ with respect to the Schauder basis $(e_k)_{k \geq 1}$. For example, this property is satisfied  in the case of weighted shifts.
\end{remark}

\subsection{Maximizing measures}
\label{maximizing-measures-section}

In this section, we will show the existence of maximizing measures, associated with a certain class of potentials, in the framework of linear dynamical systems on Banach spaces. This will be achieved under two different assumptions. First, we prove that the so-called  ground states, whose existence was obtained in section \ref{zero-temperature-limit-section}, are maximizing measures. This will be proved  for potentials $A$ satisfying the hypothesis of Theorem \ref{limit-theorem}. After that, we prove the existence  of maximizing measures under the assumption that the potential $A$ satisfies the maximizing property (see Definition \ref{maximizing-property}). In these two cases, it will follow (see Lemma \ref{bounded-mean})  that the  Ma\~n\'e potential only takes non-positive  values, which is one of the main requirements  that will be used in the proof of Theorem \ref{sub-action-theorem}.

 Now we want to prove the last claim of Theorem \ref{limit-theorem}. In this direction, the following Lemma guarantees the existence of maximizing measures under the hypotheses  of Theorem \ref{limit-theorem}.

\begin{lemma}
\label{maximizing-ground-states-lemma}
Assume that $\mu_\infty$ is a ground state obtained from the existence claimed by  Theorem \ref{limit-theorem}. Then, we have that $\mu_\infty \in \mathcal{P}_{\max}(A)$.
\end{lemma}
\begin{proof}
The proof of this Lemma follows from Theorem \ref{variational-principle} and the fact that $h_\nu(\mu) \leq 0$, for any $\mu \in \mathcal{P}_T(X)$. Indeed, it is not difficult to check that $m(A) = \lim_{t \to +\infty}\frac{\log(\lambda_{t A})}{t}$ (see for instance \cite{MR2864625, MR3377291}). Therefore, for any $\mu \in \mathcal{P}_\sigma(X)$ we have
\begin{align*}
m(A) 
= \lim_{n \in \mathbb{N}}\frac{\log(\lambda_{t_n A})}{t_n} 
&= \lim_{n \in \mathbb{N}}\frac{h_\nu(\mu_{t_nA})}{t_n} + \int_X A d\mu_{t_nA} \\
&\leq \lim_{n \in \mathbb{N}} \int_X A d\mu_{t_nA} 
= \int_X A d\mu_\infty \leq m(A) \;.
\end{align*}

The above implies that $\mu_\infty \in \mathcal{P}_{\max}(A)$.
\end{proof}

In the next Lemma, we also guarantee the existence of maximizing measures, but in this case, we assume the so-called  maximizing property.

\begin{lemma}
\label{maximizing-property-lemma}
 Consider a potential $A \in \mathcal{H}_{\alpha}(X)$ satisfying the maximizing property and let $Y, Z$ be the subspaces of $X$ satisfying \eqref{max-prop-1} and \eqref{max-prop-2}. If 
$$
Y = \mathrm{span}\{v, ... , T^{k-1}(v)\} \;,
$$ 
for some periodic point $v \in X$ of period $k$, it follows that $\mathcal{P}_{\max}(A) \neq \emptyset$.
\end{lemma}
\begin{proof}
Consider subspaces $Y, Z \subset X$ satisfying the hypothesis of this Lemma. Note that under these assumptions all the points in the subspace $Y$ are periodic points of period $k$.

Given a point $x \in X$, consider the values
$$
m_n(A, x) := \frac{1}{n}S_n A(x)
$$ 
and
$$
m(A, x) := \limsup_{n \to +\infty}m_n(A, x) \;.
$$

It is not difficult to verify that $m(A) \leq \sup_{x \in X} m(A, x)$ (see for instance \cite{MR2354972}). On other hand, $m(A, x_y) = m_k(A, x_y) \leq m(A)$, when $x_y$ belongs to $Y$. The above is true because $m(A, x_y)$ is the value of the integral of $A$ with respect to the periodic measure $\frac{1}{k}\sum_{j=0}^{k-1}\delta_{T^j(x_y)}$.

From the limit in \eqref{max-prop-2}, there is a constant $M > 0$, such that for each $x \in X$ with $\|x_y\| > M$, 
\begin{equation}
\label{aprox-ineq}
m(A, x_y) < m(A) - 1 \;.
\end{equation}

Define the set $X_M := \{x \in X :\; \|x_y\|_X \leq M\}$. Note that given a point $x \in X_M$, we can express such point as $x = x_y + x_z$ where $x_y \in Y$ has norm less than or equal to $M$ and $x_z \in Z$, i.e. $X_M = Y_M + Z$, where $Y_M := \{y \in Y :\; \|y\|_X \leq M\}$. On other hand, since any point in $Y$ is a periodic point of period $k$, it follows that
\[
Y_0 := \bigcup_{j=0}^{+\infty} T^j(Y_M) = \bigcup_{j=0}^{k-1} T^j(Y_M) \;,
\]
which implies that the set $Y_0 \subset Y$ is compact and invariant by the action of $T$.

Note that for each $x \in X$, such that, $x_y \in Y \setminus Y_0$, it is valid $\|x_y\| > M$, which implies \eqref{aprox-ineq}. Thus, taking supremum on all the points $x_y \in Y \setminus Y_0$ and using \eqref{max-prop-1}, we obtain that
\begin{equation}
\label{ineq-2}
\sup_{\substack{x \in X \\ x_y \in Y \setminus Y_0}} m(A, x) \leq \sup_{x_y \in Y \setminus Y_0} m(A, x_y) \leq m(A) - 1 \;.
\end{equation}

On other hand, we have
\begin{equation}
\label{ineq-1}
\sup_{\substack{x \in X \\ x_y \in Y_0}} m(A, x) \leq \sup_{x_y \in Y_0} m(A, x_y) \leq m(A) \;,
\end{equation}
where the first one of the inequalities is a consequence of \eqref{max-prop-1} and the second one follows from the fact that all the points belonging to the $T$-invariant set $Y_0$ are in fact periodic points of period $k$.

By inequalities \eqref{ineq-2} and \eqref{ineq-1}, we have $\sup_{x \in X} m(A, x) \leq m(A)$. Therefore, using that $m(A) \leq \sup_{x \in X} m(A, x)$ (see for instance \cite{MR2354972}), it follows that
\[
\sup_{x \in X} m(A, x) = m(A)
\]

Furthermore, by \eqref{max-prop-1}, we have
\[
\sup_{x \in X} m(A, x) = \sup_{x_y \in Y_0} m(A, x_y) = m(A)
\]

The above guarantees the existence of a maximizing measure supported on the compact set $Y_0$, which in fact is the limit in the weak* topology of periodic measures supported on periodic points belonging to $Y_0$.
\end{proof}

Note that Lemmas \ref{maximizing-ground-states-lemma} and \ref{maximizing-property-lemma} imply that $\mathcal{P}_{\max}(A)$ is a non-empty set. We claim that the existence of maximizing measures will imply that the  Ma\~n\'e potential is well-defined (i.e. there exist a finite supremum in expression 
\eqref{Mane-potential}) and only takes non-positive  values.

\begin{lemma}
\label{bounded-mean}
Assume that $\mathcal{P}_{\max}(A)$ is a non-empty set. Then, 
\begin{equation}
\label{ergodic-mean}
S_n(A - m(A)) \leq 0 \;.
\end{equation}

In particular, the Ma\~n\'e potential $\phi_A$ associated with $A$ is well-defined  and takes  values in the set $[-\infty, 0]$.
\end{lemma}
\begin{proof}
Assume that the potential $A$ satisfies the first condition of the Lemma. Since we are assuming that $\sum_{n = 1}^{+\infty}p(T^n)^{-\alpha} < +\infty$, it follows that the dynamics $T$ is Devaney chaotic, which implies that
\begin{equation}
\label{equation-sup-periodic-orbits}
\sup_{x \in X} m(A, x) = \sup_{x \in \mathrm{Per}(X)} m(A, x) \leq m(A) \;,
\end{equation}
where $\mathrm{Per}(X)$ is the set of periodic points of $X$. Actually, the last inequality in \eqref{equation-sup-periodic-orbits} follows from the fact that for any periodic orbit $\{v, ... , T^{k-1}(v)\}$ is satisfied the following
\[
\int_X A d\Bigl( \sum_{j=1}^{k-1}\delta_{T^j(v)} \Bigr) \leq m(A) \;.
\]

Thus, by \eqref{equation-sup-periodic-orbits}, we get \eqref{ergodic-mean}.
\end{proof}

\subsection{An example: weighted shifts}
\label{example-section}

In this section, we present some particular cases where the assumptions of  Theorems \ref{variational-principle} and \ref{limit-theorem} are true. We will consider the so-called  weighted shifts in this section (see for details \cite{MR2533318, MR3759568, BeMe20, MR2919812, LMSV19}).

Consider the normed space $c_0(\mathbb{R})$ which is the set of sequences of real numbers $(x_n)_{n \geq 1}$ satisfying $\lim_{n \to +\infty} x_n = 0$ equipped with the norm
$$
\|x\|_{c_0(\mathbb{R})} := \sup_{n \geq 1} |x_n| \;,
$$
and the normed space $l^p(\mathbb{R})$, $1 \leq p < +\infty$, which is the set of  sequences satisfying $\sum_{n=1}^{+\infty} |x_n|^p < +\infty$ equipped with the norm
$$
\|x\|_{l^p(\mathbb{R})} := \bigl(\sum_{n=1}^{+\infty} |x_n|^p\bigr)^{\frac{1}{p}} \;.
$$

{It is widely known that $c_0(\mathbb{R})$ and $l^p(\mathbb{R})$, $1 \leq p < +\infty$, are separable Banach spaces which are equipped with the Schauder basis $(e_k)_{k \geq 1}$, where $e_k = (\delta_{ik})_{i \geq 1}$. Moreover, in the case $1 < p < +\infty$, the space $l^p(\mathbb{R})$ is reflexive as well.

Consider $c, c' \in \mathbb{R}$ satisfying $0 < c < c'$ and a sequence $(\alpha_n)_{n \geq 1}$, such that, $\alpha_n \in (c, c')$ for each $n \in \mathbb{N}$. For each pair of natural numbers $k, n$ define
$$
\beta_k^n := \alpha_k ... \alpha_{k+n-1} \;.
$$

Now, fixing the number $n$, define
$$
d_n := \inf\{\beta_k^n :\; k \in \mathbb{N}\} \;.
$$

The weighted shift $L : X \to X$, where $X = c_0(\mathbb{R})$ or $X = l^p(\mathbb{R})$, $1 \leq p < +\infty$, is given by the linear map $L((x_n)_{n \geq 1}) = (\alpha_n x_{n+1})_{n \geq 1}$. It is well-known that the asymptotic behavior of the values $\beta_k^n$ characterize the topological properties of the weighted shift $L$ (see Remark $1$ in \cite{LMSV19} and \cite{MR2533318, MR3759568, BeMe20, MR2919812} for details).  Indeed, taking $p(L^n) = d_n$, for each $n \in \mathbb{N}$, we get that  the map $p$ satisfies \eqref{p-norm}. Thus, the convergence of the series $\sum_{n=1}^{+\infty} d_n^{-\alpha}$ to a real number implies that the weighted shift $L$ is Devaney chaotic and topologically transitive.

Under the assumptions above, the Ruelle operator defined in \eqref{Ruelle-operator} is given by
\[
\mathcal{L}_A(\varphi)(x_1, x_2,\; ... ) := \int_{\mathbb{R}} e^{A\bigl(r, \frac{x_1}{\alpha_1}, \frac{x_2}{\alpha_2},\; ... \bigr)} \varphi\bigl(r, \frac{x_1}{\alpha_1}, \frac{x_2}{\alpha_2},\; ... \bigr) d\nu(r) \;.
\]

The next Propositions are a consequence of the main Theorems of this paper; under some strong assumptions in the values $\beta_k^n$ and $d_n$.

\begin{proposition}
\label{variational-principle-ws}
Let $X$ be $c_0(\mathbb{R})$ or $l^p(\mathbb{R})$, $1 \leq p < +\infty$, and $L : X \to X$ a weighted shift. Assume also that one of the following conditions is satisfied:
\begin{enumerate}[i)]
\item $\lim_{n \to +\infty} (d_n)^{-\frac{1}{n}}$;
\item $\sup\bigl\{ \sum_{n = 1}^{+\infty} (\beta_k^n)^{-1} : k \in \mathbb{N} \bigr\} < +\infty$.
\end{enumerate}

Then, for each bounded above potential $A \in \mathcal{H}_{\alpha}(X) \cap \mathcal{SV}_L(X)$ the Gibbs state $\mu_A$ associated with $A$ satisfies
\[
\log(\lambda_A) = h_\nu(\mu_A) + \int_X A d\mu_A = P_\nu(A) \,.
\]
\end{proposition}
\begin{proof}
 By Lemma $2$ in \cite{LMSV19}, items $i)$ and $ii)$ are equivalent to the property $\sum_{n = 1}^{+\infty} (d_n)^{-\alpha} < +\infty$. Thus, our claim follows from Theorem \ref{variational-principle}, because $p(L^n) = d_n$, for each $n \in \mathbb{N}$.
\end{proof}

\begin{proposition}
\label{limit-theorem-ws}
Let $X$ be $c_0(\mathbb{R})$ or $l^p(\mathbb{R})$, $1 \leq p < +\infty$, and $L : X \to X$ a weighted shift. Assume also one of the following conditions:
\begin{enumerate}[i)]
\item $\lim_{n \to +\infty} (d_n)^{-\frac{1}{n}}$;
\item $\sup\bigl\{ \sum_{n = 1}^{+\infty} (\beta_k^n)^{-1} : k \in \mathbb{N} \bigr\} < +\infty$.
\end{enumerate}

Then, for any bounded above potential $A \in \mathcal{H}_\alpha(X) \cap \mathcal{SV}_L(X)$, satisfying the summability condition, the family of equilibrium states $(\mu_{tA})_{t>1}$ has accumulation points at infinity which belongs to the set $\mathcal{P}_{\max}(A)$.
\end{proposition}
\begin{proof}
The proof follows  the same reasoning already used  in the proofs of Proposition \ref{variational-principle-ws} and Lemma \ref{maximizing-property-lemma}.
\end{proof}

\begin{remark}
Note that Proposition \ref{limit-theorem-ws} implies the existence of maximizing measures in the setting of weighted shifts which, in particular, will guarantee   that the  Ma\~n\'e potential is well-defined, under our hypothesis, as a consequence of Lemma \ref{bounded-mean}.
\end{remark}

\section{Existence of sub-actions and some examples}
\label{sub-actions-section}

In this section Theorem \ref{sub-action-theorem} will be proved. We also present some properties of the  Ma\~n\'e potential and we show  some explicit examples where calibrated sub-actions and maximizing measures do exist (in the setting weighted shifts). Moreover, we show explicit examples where there exists  selection at zero temperature. The first of our examples concerns  the existence of a  calibrated sub-actions in the framework of weighted shifts on the space $l^1(\mathbb{R})$. The other examples concern the uniqueness of the maximizing measure for potentials defined on weighted shifts in the spaces $c_0(\mathbb{R})$ and $l^p(\mathbb{R})$, $1 \leq p < +\infty$.

\subsection{The Ma\~n\'e potential}
\label{sub-actions-subsection}

Here we present the proof of Theorem \ref{sub-action-theorem} and also some results concerning the behavior of the Ma\~n\'e potential. We will adapt for the Linear dynamics framework results from \cite{GL1}. We assume throughout this section that $X$ is a separable Banach space, $T : X \to X$ is a bounded linear operator surjective, but not bijective. We also assume that for each value of $n$, we get $X = \mathrm{Ker}(T^n) \oplus E_n$, with $T(E_{n+1}) \subset E_n$ for each $n \in \mathbb{N}$, $0 < \dim(\mathrm{Ker}(T)) < +\infty$, and $\sum_{n = 1}^{+\infty}p(T^n)^{-\alpha} < +\infty$.

\begin{proof}[Proof of Theorem \ref{sub-action-theorem}]
Fix a certain point $x \in X$. We want to prove that the map $\phi_A(x, \cdot)$ is a sub-action. In order to do that, fixing another point $y \in X$, we consider $\epsilon > 0$, $n \in \mathbb{N}$, and $x' \in T^{-n}(y)$, such that, $\| x - x'\|_X < \epsilon$. Then, we get
\begin{equation}
\label{ergodic-sums}
S_{n+1}(A - m(A))(x') = S_n(A - m(A))(x') + A(y) - m(A) \;.
\end{equation}

Since $T^{n+1}(x') = T(y)$,  taking supremum in \eqref{ergodic-sums}, first in the left side, and, after that in the right side, we obtain
\[
\sup_{m \in \mathbb{N}}\sup_{\substack{x' \in T^{-m}(T(y)) \\ \|x - x'\|_X < \epsilon}} S_m(A - m(A))(x')
\geq \sup_{n \in \mathbb{N}}\sup_{\substack{x' \in T^{-n}(y) \\ \|x - x'\|_X < \epsilon}} S_n(A - m(A))(x') + A(y) - m(A) \;.
\]

Taking the limit, when $\epsilon$ goes to $0$, in both sides of the above inequality, it follows that
\[
\phi_A(x, T(y)) \geq \phi_A(x, y) + A(y) - m(A) \;.
\]

That is, the map $\phi_A(x, \cdot)$ is a sub-action.

On other hand, taking supremum in \eqref{ergodic-sums}, first on the right side, and after that on the left side, and subsequently, taking the limit when $\epsilon$ goes to $0$, we obtain that
\[
\phi_A(x, T(y)) \leq \phi_A(x, y) + A(y) - m(A) \;.
\]

Therefore, $\phi_A(x, \cdot)$ is a calibrated sub-action. The above shows that the potential $A$ is cohomologous to the constant $m(A)$, via the Ma\~n\'e potential.

On other hand, we will get that $\phi_A(x, \cdot) \in \mathcal{H}_\alpha(X)$, as a consequence of the assumption $A \in \mathcal{H}_\alpha(X)$. Indeed, given $\epsilon > 0$ and $y^1, y^2 \in X$, it follows that for each $n \in \mathbb{N}$, 
$x^1 \in T^{-n}(y^1)$ and $x^2 \in T^{-n}(y^2)$, satisfying  $\|x - x^1\|_X < \epsilon$ and $\|x - x^2\|_X < \epsilon$, we have

\begin{align*}
S_n(A - m(A))(x^1)
&\leq S_n(A - m(A))(x^2) + \sum_{j = 0}^{n - 1}|A(T^j(x^1)) - A(T^j(x^2))|  \\
&\leq S_n(A - m(A))(x^2) + \mathrm{Hol}^{\alpha}_A \sum_{j = 0}^{n - 1}\|T^j(x^1) - T^j(x^2)\|_X^{\alpha}  \\
&\leq S_n(A - m(A))(x^2) + \mathrm{Hol}^{\alpha}_A \Bigl(\sum_{j = 1}^{n} p(T^j)^{-\alpha} \Bigr) \|y^1 - y^2\|_X^{\alpha} \\
&\leq S_n(A - m(A))(x^2) + \mathrm{Hol}^{\alpha}_A \Bigl(\sum_{n = 1}^{+\infty} p(T^n)^{-\alpha} \Bigr) \|y^1 - y^2\|_X^{\alpha}  \;.
\end{align*}

Thus, taking the supremum on $x^1$ and $x^2$, after that, taking the supremum on all the natural numbers $n$, and, finally, taking the limit when $\epsilon$ goes to $0$ - in the first and in the second one - on the above expressions, we get
\[
\phi_A(x, y^1) \leq \phi_A(x, y^2) + \mathrm{Hol}^{\alpha}_A \Bigl(\sum_{n = 1}^{+\infty} p(T^n)^{-\alpha} \Bigr) \|y^1 - y^2\|_X^{\alpha} \;.
\]

In an analogous way, we can also prove that
\[
\phi_A(x, y^2) \leq \phi_A(x, y^1) + \mathrm{Hol}^{\alpha}_A \Bigl(\sum_{n = 1}^{+\infty} p(T^n)^{-\alpha} \Bigr) \|y^1 - y^2\|_X^{\alpha} \;.
\]

The above implies that $\phi_A(x, \cdot) \in \mathcal{H}_{\alpha}(X)$, with H\"older constant less or equal than $\mathrm{Hol}^{\alpha}_A \Bigl(\sum_{n = 1}^{+\infty} p(T^n)^{-\alpha} \Bigr)$. In particular, $\phi_A(x, y) \in (-\infty, 0)$, for each $y \in X$, and this is the end of the proof.
\end{proof}

Now, we present an interesting property of the Ma\~n\'e potential. The proof is valid on the general framework  of Linear Dynamics on Banach spaces. It shows a relation of the Ma\~n\'e potential with a general sub-action $V$.
Analogous results on the framework of classical ergodic optimization appeared in \cite{MR1855838} and \cite{GL1}.

\begin{proposition}
Suppose that the potential $A \in \mathcal{H}_\alpha(X)$ and $\mathcal{P}_{\max}(A)$ is a non-empty set. Let $V \in \mathcal{H}_{\alpha}(X)$ be a general sub-action for $A \in \mathcal{H}_{\alpha}(X)$, and $\phi_A$ the Ma\~n\'e potential for $A$. Then, for any pair of points $x, y \in X$ we get
\begin{equation}
\label{Mane-sub-action}
\phi_A(x, y) \leq V(y) - V(x) \;.
\end{equation}

\end{proposition}
\begin{proof}
Fix the points $x, y \in X$. First note that for any point $x' \in X$, we have
\[
|V(x) - V(x')| \leq \mathrm{Hol}^{\alpha}_V \|x - x'\|_X^{\alpha} \;.
\]

Moreover, since $V$ is sub-action, it follows that for each $n \in \mathbb{N}$  it is also satisfied the inequality
\[
V(T^n(x')) - V(x') \geq S_n(A - m(A))(x') \;.
\]

Then, given $\epsilon > 0$, $n \in \mathbb{N}$ and $x' \in T^{-n}(y)$, such that, $\|x - x'\|_X < \epsilon$, we have
\begin{align*}
V(y) - V(x)
&= V(T^n(x')) - V(x)  \\
&\geq V(T^n(x')) - V(x') - \mathrm{Hol}^{\alpha}_V \|x - x'\|_X^{\alpha}  \\
&> S_n(A - m(A))(x') - \mathrm{Hol}^{\alpha}_V \epsilon^{\alpha}  \;.
\end{align*}

Therefore, taking the supremum among all the points $x' \in T^{-n}(y)$, satisfying that $\|x - x'\|_X < \epsilon$, and,  after that, taking the supremum among all  natural numbers $n$, and, finally, taking the limit when $\epsilon$ goes to $0$, we get \eqref{Mane-sub-action}.
\end{proof}

\subsection{Explicit examples of sub-actions and uniqueness of maximizing measure}
\label{examples}

In here, we present explicit examples of a calibrated sub-actions associated with bounded above H\"older continuous potentials. In some examples,  we guaranteed selection  at temperature zero temperature.

\begin{example} \label{maine}
 For fixed $c>1$, consider the operator $L: l^1(\mathbb{R}) \to l^1(\mathbb{R})$ given by
$$
L((x_n)_{n \geq 1}) := (c\, x_{n+1})_{n \geq 1}\;.
$$

Taking $v := (c^{-n+1})_{n \geq 1}$, it follows that for any $\alpha \in \mathbb{R}$ we get
$$
L(\alpha v) =  \alpha L(v) = \alpha v\;.
$$

That is, $\alpha v$ is a fixed point for $L$  for each $\alpha \in \mathbb{R}$. Consider the unbounded H\"older continuous potential
$$
A(x) := -\,\, \|x - v\|_{l^1(\mathbb{R})} \,.
$$

It follows that the delta Dirac measure $\delta_{ v}$ is the unique maximizing probability measure for $A$ and $m(A) = 0$.

We want to find a sub-action associated with the potential $A$ i.e. a function $V$ such that for all $x$ is satisfied
\begin{equation} \label{VV}
V(L(x)) \geq A(x) + V(x).
\end{equation}

Assuming that $c = 2$ and taking $V(x) := A(x)$, for each $x \in l^1(\mathbb{R})$, we want to show that (\ref{VV}) is true.  Note that for $x = (x_n)_{n \geq 1} \in l^1(\mathbb{R})$ we have
$$
\|x - v\|_{l^1(\mathbb{R})} = \sum_{n=1}^{+\infty}|x_n - 2^{-n+1}|
$$
and
$$
\|L (x) - v\|_{l^1(\mathbb{R})} = \sum_{n=1}^{+\infty} |2\, x_{n+1} - 2^{-n+1}| = 2\sum_{n=1}^{+\infty} |x_{n+1} - 2^{-n}| \;.
$$

Then, $V$ is a sub-action for $A$, because
\begin{align*}
V(L(x)) = -\|L (x) - v\|_{l^1(\mathbb{R})}
&= -2\sum_{n=1}^{+\infty} |x_{n+1} - 2^{-n}|  \\
&\geq -2\sum_{n=1}^{+\infty} |x_{n} - 2^{-n + 1}|  \\
&= -2\|x - v\|_{l^1(\mathbb{R})} = A(x) + V(x)  \;.
\end{align*}

Furthermore, this sub-action $V$  is calibrated. Indeed, given a sequence $y = (y_n)_{n \geq 1}$ in $l^1(\mathbb{R})$, the $L$-preimages $x$ of the point $y$ are of the form
$$
x = (x_n)_{n \geq 1} = \bigl(x_1, \frac{y_1}{2},  \frac{y_2}{2}, ... \bigr) \;.$$

Taking $\widetilde{x} := \bigl(1, \frac{y_1}{2} ,  \frac{y_2}{2}, ... \bigr)$, we get
\begin{align*}
V(y) = V(L(\widetilde{x}))
&= -\| L(\widetilde{x}) - v\|_{l^1(\mathbb{R})}  \\
&= - \sum_{n=1}^{+\infty} |y_n - 2^{-n+1}|  \\
&= -2 \sum_{n=1}^{+\infty} |\frac{y_n}{2} - 2^{-n}|  \\
&= -2\,(\,|1 - 1| \,+\,|\frac{y_1}{2} - 2^{-1}|\,+\, |\frac{y_2}{2} - 2^{-2}| + ... )  \\
&= -2 \, \|\widetilde{x} - v\|_{l^1(\mathbb{R})} = A(\widetilde{x}) + V(\widetilde{x})  \;.
\end{align*}
\end{example}
\smallskip

The next examples present cases in which it can be guaranteed selection at zero temperature - for a H\"older continuous potential - defined on the spaces of sequences $c_0(\mathbb{R})$ and $l^p(\mathbb{R})$, $1 \leq p < +\infty$.

\begin{example}
For fixed $c > 1$ and either $X = l^p(\mathbb{R})$ or $X = c_0(\mathbb{R})$, consider the operator $L: X \to X$ given by
$$
L((x_n)_{n \geq 1}) := (c\, x_{n+1})_{n \geq 1}\;.
$$

Taking $v := (c^{-n+1})_{n \geq 1}$, it follows  $\alpha v$ is a fixed point for each $\alpha \in \mathbb{R}$.

Consider a monotonous decreasing $1$-H\"older continuous function $r$ from $[0,1]$ into $\mathbb{R}$, such that, $r(0) = 1$ and suppose  $\lim_{s \to +\infty}r(s) = 0$. Consider the potential
$$
A(x) := -\, r(\|x - v\|_X)\,\|x - v\|_X \;.
$$

It follows that $A \in \mathcal{H}_1(X)$,  moreover, the delta Dirac measure $\delta_{ v}$ is the unique maximizing probability measure for $A$ and $m(A) = 0$. Therefore, there exists selection of probability at zero temperature.
\end{example}

\begin{example}
 For fixed $c_0,c_1 > 1$ and either $X = l^p(\mathbb{R})$ or $X = c_0(\mathbb{R})$, consider the operator $L: X \to X$ given by
$$
L((x_n)_{n \geq 1}) := (c_i\, x_{n+1})_{n \geq 1} \;,
$$
where $i \in \{0, 1\}$ is the unique number such that $n = 2k + i$ for some $k \in \mathbb{N} \cup \{0\}$. Then,
$$
L^2((x_n)_{n \geq 1}) = (c_0\,c_1\, x_{n+2})_{n \geq 1} \;.
$$

Taking
$$
v := (1, 0, (c_0 \, c_1)^{-1}, 0,  (c_0\, c_1)^{-2},0,  (c_0 \, c_1)^{-3},0, ... )
$$
and
$$
w := (0, c_0\, (c_0 \, c_1)^{-1}, 0, c_0\, (c_0 \, c_1)^{-2},0, c_0\, (c_0 \, c_1)^{-3}, 0, ... ) \;,
$$
it follows that $L(v) = w$ and $L^2(v) = v$ i.e. $v$ and $w$ are periodic points of period two for the map $L$. Given $\alpha_1,\alpha_2$, the points of the form $\alpha_1 v + \alpha_2 w$ are points of period two, which generate a linear subspace of periodic points of period two.

Consider the potential
$$
A(x) := -[r(\,\|x - v\|_X\,) \|x - v\|_X + r(\,\|x - w\|_X\,)\|x - w\|_X] \;,
$$
where $r$ was defined above.

We have that $A \in \mathcal{H}_1(X)$, the  periodic orbit $\{v, w\}$ of period two is such that $1/2\, \delta_v + 1/2 \,\delta_w$, and it is the unique  maximizing probability measure for the potential $A$ and $m(A) = 0$.
\end{example}

\begin{example}
For fixed $c>1$, and either for $X = l^p(\mathbb{R})$ or $X = c_0(\mathbb{R})$, consider the operator $L: X \to X$ given by
$$
L((x_n)_{n\geq1}) := (c\, x_{n+1})_{n\geq1} \;.
$$

Denote $W \subset X$
the infinite-dimensional vector subspace  of the form
$$
W :=\{\ (x_1, 0, x_2, 0, ... , 0, x_n, 0, ... ) :\; x \in X\} \;.
$$

Define the potential
$$
A(x) := -r(d_X (x, W))  \,d_X(x, W) \;,
$$
where $r$ was defined above and $d_X(x, W) := \inf\{\|x - w\|_X :\; w \in W\}$. Then, $A \in \mathcal{H}_1(X)$ and any $\mu \in \mathcal{P}_L(X)$, with support contained in $W$, is maximizing for $A$. For instance, the probability measure $t\delta_v + (1-t)\delta_{L(v)}$, $t \in (0, 1)$, where
$$
v := (1, 0, c^{-2}, 0, c^{-4}, 0, ... , 0, c^{-2n}, 0, ... ) \;,
$$
is a maximizing measure for the potential $A$.
\end{example}

\end{document}